\documentclass[12pt,compress]{amsart}

\usepackage[numbers,sort&compress]{natbib}

\usepackage{amsmath}
\usepackage{amsthm}
\usepackage{amssymb}
\usepackage{mathrsfs}
\usepackage{scrextend}
\usepackage{enumerate}
\usepackage{longtable,booktabs,makecell,multirow,tabularx}
\usepackage{graphicx}
\usepackage{latexsym}
\usepackage{amssymb}
\usepackage{color}
\usepackage{enumitem}

\usepackage{hyperref}

\usepackage{geometry}
\newtheorem{theorem}{Theorem}[section]
\newtheorem{lemma}[theorem]{Lemma}
\newtheorem{corollary}[theorem]{Corollary}

\newtheorem{theoremA}{Theorem}

\newtheorem{theoremone}{Theorem}

\theoremstyle{definition}

\newtheorem{exampleone}[theoremone]{Example}
\newtheorem{definitionone}[theoremone]{Definition}

\theoremstyle{remark}

\numberwithin{equation}{section}

\makeatletter
\@namedef{subjclassname@2020}{\textup{2020} Mathematics Subject
	Classification}
\makeatother

\def\UU{{\mathfrak U}}
\def\FF{{\mathfrak F}}

\newcommand{\R}{\mathbb{\R}}

\usepackage{pifont}

\begin{document}
	\begin{sloppypar}
	

\title[On the $IC$-$\Pi$-property of subgroups of finite groups]{On the $ IC $-$ \Pi $-property of subgroups of finite groups}

\author[Z. Qiu]{Zhengtian Qiu}
\address[Zhengtian Qiu]{School of Mathematics and Statistics, Guangdong University of Technology, Guangzhou 510520, People's Republic of China}
\email{qztqzt506@163.com}

\author[S. Qiao]{Shouhong Qiao\textsuperscript{$ \ast $}}
\address[Shouhong Qiao]{School of Mathematics and Statistics, Guangdong University of Technology, Guangzhou 510520, People's Republic of China}
\email{qshqsh513@163.com
}

\thanks{\textsuperscript{$ \ast $}Corresponding author}

\date{}

\subjclass[2020]{20D10, 20D20}

\keywords{Finite group, $ IC $-$ \Pi $-property, $ p\mathfrak{U} $-hypercenter, $p$-supersoluble group}


\maketitle

\begin{abstract}
    Let $ H $ be a subgroup of a finite group $ G $. We say that $ H $ satisfies the $ \Pi $-property in $ G $ if $ | G/K : N_{G/K}((H \cap L)K/K)| $ is a $ \pi(( H \cap L)K/K )  $-number for any chief factor $ L/K $ of  $ G $; and we call that $ H $  satisfies the $ IC $-$ \Pi  $-property in $ G $ if  $ H\cap [H, G] $  satisfies the $ \Pi  $-property in $ G $. In this paper, we obtain a criterion of a normal subgroup  being contained in the $ p\mathfrak{U} $-hypercenter of a finite group by the  $IC$-$ \Pi $-property of some $ p $-subgroups.
\end{abstract}


    	
    	


\section{Introduction}

%


  All groups considered in this paper will be finite. Throughout the paper, $ G $ always denotes a finite group, $ p $ a fixed prime, $ \pi $
  a set of primes and $ \pi(G) $ the set of all primes dividing the group order $ |G| $. An integer $ n $ is called a \emph{$ \pi $-number}
  if all prime divisors of $ n $ belong to $ \pi $. In particular, an integer is called a $p$-number if it is a power of $p$.

  Let $ H $ be a subgroup of $ G $.  As everyone knows,   the normal closure $  H^{G} $ of $ H $ in $ G $ is the smallest normal subgroup of $ G $ containing $ H $ and it is equal to the product of $ H $ and $ [H, G] $, that is, $ H^{G}= H[H, G] $, where $ [H, G] $ is the commutator subgroup of $ H $ and $ G $. We denote a subgroup $ H $ intersecting the commutator subgroup of $ H $ and $ G $ by $ IC $, and we say that  a subgroup $ H $ has the \emph{ $ IC $-property } if the intersection of $ H $ and $ [H, G] $ satisfy some conditions. It is an interesting question to study  the relationship between $ IC $-properties of some subgroups and the structure of a finite group. Some  significant results had been given (for example, see \cite{Gao-Li-2021, Gao-Li-2023,Kaspczyk}).

  In 2011, Li \cite{li-2011} introduced  the definition of the $ \Pi $-property of
  subgroups of finite groups, which generalized many embedding properties of subgroups (see \cite[Section 2]{li-2011}). Let $ H $ be a subgroup of $ G $. We say that $ H $  satisfies the \emph{$ \Pi  $-property} in $ G $ if for any   chief factor $ L / K $ of $ G $, $ | G / K : N _{G / K} ( HK/K\cap L/K )| $ is a $ \pi (HK/K\cap L/K) $-number.
  In this paper, we introduce the following concept which combines the  $ \Pi $-property and $ IC $-property.

  \begin{definitionone}
  	Let $ H $ be a subgroup of $ G $. We say that $ H $  satisfies the \emph{$ IC $-$ \Pi  $-property}  in $ G $ if  $ H \cap [H, G] $  satisfies the $ \Pi  $-property in $ G $.
  \end{definitionone}

  Suppose that $ H $ is  a $ p $-subgroup of $ G $. If $ H $ satisfies the $ \Pi  $-property in $ G $, then $ H $ satisfies the $ IC $-$ \Pi  $-property  in $ G $ (see \cite[Lemma 2.1(6)]{Li-Miao}). However, the converse does not hold in general. Let us  see  the following example.

  \begin{exampleone}
  	Let  $ X $ be a non-abelian simple group such that a Sylow $ p $-subgroup of $ X $ has order $ p $. Let $ x $ be an element of $ X $ of order $ p $. Let $ Y=\langle y|y^{p}=1\rangle $ be a cyclic group of order $ p $.  Write $ G=X\times Y $. Then $ H=\langle xy \rangle $ has order $ p $. It is easy to see that $ H\cap [H, G]=1 $ satisfies the $\Pi$-property in $G$. For the $G$-chief factor $ G/Y $, we see that $ |G/Y:N_{G/Y}(HY/Y\cap G/Y)| $ is not a $p$-number. Hence $ H $ does not  satisfy the $\Pi$-property in $G$.
  \end{exampleone}

  As usual, we use $ \UU $ (respectively, $ p\UU $) to denote the class of supersoluble (respectively, $ p $-supersoluble) groups.   Let $Z_{\UU}(G)$ (respectively, $Z_{p\UU}(G)$ denote the \emph{$\UU$-hypercenter}  (respectively, \emph{$p\UU$-hypercenter}) of $G$, that is, the product of all normal subgroups $H$ of $G$ such that all chief factors (respectively, $p$-chief factors)  of $G$ below $H$ are cyclic.  For the \emph{generalized Fitting subgroup}  $ F^{*}(G) $ and the \emph{generalized $ p $-Fitting subgroup}
  $ F^{*}_{p}(G) $ of $ G $, please refer to \cite[Chapter 9]{Isaacs-2008}  and \cite{Adolfo-2009}.

   In this paper, we obtain a criterion of a normal subgroup $ N $ of $ G $ being contained in $ Z_{p\UU}(G) $ in terms of the $ IC $-$ \Pi $-property of some $ p $-subgroups. Our main result is as follows.

  \begin{theoremA}\label{main-result}
  	Let $ N $ be a normal subgroup of  $ G $ such that $ p $ divides the order of $ N $, and
  	$ X $ a normal subgroup of $ G $ satisfying $ F_{p}^{*}(N) \leq X \leq N $. Then $ N\leq Z_{p\UU}(G) $ if there exists a Sylow $ p $-subgroup $P$ of $ X $ satisfies the following:
  	
  	\begin{enumerate}[fullwidth,itemindent=1em,label=\rm{(\arabic*)}]
  		\item\label{1111} Every subgroup  of $ P $ of order $ d $  satisfies the  $ IC $-$ \Pi $-property in $ G $, where $ d $ is a
  		power of $ p $ with $ 1 < d < |P| $;
  		\item If $ p =d = 2 $ and $ P $ is non-abelian,  we further suppose that every  cyclic subgroup  of $ P $ of order $ 4 $ satisfies the $ IC $-$ \Pi $-property in $ G $.
  		
  	\end{enumerate}	
  	
  \end{theoremA}


The proof of Theorem \ref{main-result} consists of  a large number of steps. The following results are the main stages of it.

  \begin{theoremA}\label{order-d}
  	Let $ P $ be  a normal $ p $-subgroup of $ G $. Then $ P\leq Z_{\UU}(G) $ if $ P $ satisfies the following:
  	
  	\begin{enumerate}[fullwidth,itemindent=1em,label=\rm{(\arabic*)}]
  		\item Every subgroup  of $ P $ of order $ d $  satisfies the  $ IC $-$ \Pi $-property in $ G $, where $ d $ is a
  		power of $ p $ with $ 1 < d < |P| $;
  		\item If $ p =d = 2 $ and $ P $ is non-abelian,  we further suppose that every  cyclic subgroup  of $ P $ of order $ 4 $ satisfies the $ IC $-$ \Pi $-property in $ G $.
  		
  	\end{enumerate}
  \end{theoremA}


  \begin{theoremA}\label{minimal}
  	Let  $ N $ be a normal subgroup of $ G $ and $ P\in {\rm Syl}_{p}(N) $.
  	Then $ N\leq Z_{\UU}(G) $ if $ P $ satisfies the following:
  	
  	\begin{enumerate}[fullwidth,itemindent=1em,label=\rm{(\arabic*)}]
  		\item Every subgroup  of $ P $ of order $ p $  satisfies the  $ IC $-$ \Pi $-property in $ G $;
  		\item If $ p = 2 $ and $ P $ is non-abelian,  we further suppose that every  cyclic subgroup  of $ P $ of order $ 4 $ satisfies the $ IC $-$ \Pi $-property in $ G $.
  		
  	\end{enumerate}
  	
  \end{theoremA}

\begin{theoremA}\label{maximal}
	Let  $ N $ be a normal subgroup of $ G $ and $ P\in {\rm Syl}_{p}(N) $.
	If every maximal subgroup of $ P $ satisfies the  $ IC $-$ \Pi $-property in $ G $, then either $ |P| = p $ or $ N\leq Z_{p\UU}(G) $.
	
\end{theoremA}





\section{Preliminaries}

In this section, we give some lemmas that will be used in our proofs.


  \begin{lemma}[{\cite[Lemma 2.2]{li-2011}}]\label{over}
  	 Let $ H \leq G $ and $ N \unlhd G $. If $ H $ satisfies the  $ \Pi $-property in $ G $, then $ HN/N $ satisfies the  $ \Pi $-property in $ G $.
  		
  \end{lemma}

\begin{lemma}\label{OVE}
	Let $ N $ be a normal subgroup of $ G $. Assume that $ H\leq G $ and  $ H $ satisfies the  $ IC $-$ \Pi $-property in $ G $. Then
	
	\begin{enumerate}[fullwidth,itemindent=1em,label=\rm{(\arabic*)}]
		\item\label{i} If $ N \leq H $, then $ H/N $ satisfies the  $ IC $-$ \Pi $-property in $ G $.
		\item\label{ii}  If $ (|H|,|N|) = 1 $, then $ HN/N $ satisfies the  $ IC $-$ \Pi $-property in $ G $.
	\end{enumerate}
	
\end{lemma}

\begin{proof}
	
	\ref{i}  Clearly, $ H/N \cap [H/N, G/N] = (H \cap [H, G])N/N $. By Lemma \ref{over}, we see that $ H/N \cap [H/N, G/N] = (H \cap [H, G])N/N $ satisfies the $ \Pi $-property in $ G/N $.
	 This means that  $ H/N $ satisfies the  $ IC $-$ \Pi $-property in $ G/N $.
	
	\ref{ii} Since $ (|H|,|N|) = 1 $, we have  $$
    	\begin{aligned}
    	|HN \cap [H, G]| &= \frac{|HN||[H, G]|}{|HN[H, G]|} \\
    	&=\frac{|N\cap H[H, G]||H||[H, G]|}{|H[H, G]|}\\
    	&=|H\cap [H, G]||N\cap H[H, G]|\\
    	&=|H\cap [H, G]||N\cap [H, G]|.
    	\end{aligned}
    	$$
It follws that $  HN \cap [H, G] = (H \cap [H, G])(N \cap H[H, G]) $, and so $ HN \cap  [HN, G]N = HN \cap [H, G]N = (HN \cap [H, G])N = (H \cap [H, G])N $. This yields that $$ HN/N \cap [HN/N, G/N] = (HN \cap [HN, G]N)/N = (H \cap [H, G])N/N . $$ By Lemma \ref{over}, $ H/N $ satisfies the  $ IC $-$ \Pi $-property in $ G/N $.
\end{proof}

\begin{lemma}\label{satifies}
	Let $ T $ be a minimal normal subgroup of $ G $ and $ L\leq G $. Assume that $ | T | = | L | = p $. If $ LT $  satisfies the $ IC $-$ \Pi $-property in $ G $, then $ L $ also  satisfies the $ IC $-$ \Pi $-property in $ G $.
\end{lemma}

\begin{proof}
	If $ L=T $, there is nothing to prove. Hence we can assume that $ LT $ has order $ p^{2} $. If $ L \cap [L, G]=1 $, then we are done. Hence we may assume that  $ L\cap [L, G]=L $, that is, $  L \leq [L, G]\leq [LT, G]$. It follows that $ L\leq LT\cap [LT, G]\leq LT $. By hypothesis, $ LT\cap [LT, G] $ satisfies the $ \Pi $-property in $ G $. If $ LT \cap [LT, G]=L $, then $ L $ satisfies the $ \Pi $-property in $ G $. In particular,  $ L $   satisfies the $ IC $-$ \Pi $-property in $ G $, as wanted.  So we can assume that  $ LT \cap [LT, G]=LT $. This yields that $ LT $ satisfies the $ \Pi $-property in $ G $. By \cite[Lemma 2.7]{Su-2014}, we deduce that $ L $   satisfies the $ \Pi $-property in $ G $, and so  $ L $   satisfies the $ IC $-$ \Pi $-property in $ G $, as wanted.
\end{proof}






Let $ P $ be a $ p $-group. Set $ \Omega(P)=\langle  x\in P\,|\, x^{4}=1 \rangle$ if $ P $ is a non-abelian $ 2 $-group;
set $ \Omega(P) =\langle x\in P\,|\, x^{p}=1 \rangle$, otherwise.


\begin{lemma}[{\cite[Lemma 4.3]{Guo-JA}}]\label{exp}
	 Let $ C $ be a Thompson critical subgroup (see \cite[p. 185]{Gorenstein-1980}) of a nontrivial
	$ p $-group $ P $.
	
	\begin{enumerate}[fullwidth,itemindent=1em,label=\rm{(\arabic*)}]
		\item If $ p $ is odd, then the exponent of  $ \Omega(C) $ is $ p $.
		\item If $ p = 2 $, then the exponent of $ \Omega(C) $ is at most $ 4 $. Moreover, if $ P $ is an abelian $ 2 $-group, then the exponent of $ \Omega(C) $ is $ 2 $.
	\end{enumerate}
\end{lemma}

A class $ \FF $ of groups is called a \emph{formation}  if $ \FF $ is closed under taking homomorphic images and subdirect products.  A formation $ \FF $ is said to be \emph{saturated}  if $ G\in \FF $ whenever $ G/\Phi(G) \in \FF $.

\begin{lemma}[{\cite[Lemma 4.4]{Guo-JA}}]\label{then}
	Let $ \FF $ be a saturated formation, $ P $ be a normal $ p $-subgroup of  $ G $ and $ D = \Omega(C) $, where $ C $ is a Thompson critical subgroup of $ P $. If $ C \leq Z_{\FF}(G) $ or $ D \leq  Z_{\FF}(G) $,
	then $ P \leq  Z_{\FF}(G) $.
\end{lemma}


 \begin{lemma}[{\cite[Lemma 2.3]{Shen-Qiao}}]\label{Phi}
 	 Let $ P $ be a normal $ p $-subgroup of $ G $. Suppose that $ P/\Phi(P) \leq Z_{\UU}(G/\Phi(P)) $. Then $ P \leq Z_{\UU}(G) $.
 \end{lemma}

\begin{lemma}[{\cite[Chapter A, Lemma 2.1]{Doerk}}]\label{equivalent}
	Let $ U $, $ V $ and $ W $ be subgroups of $ G $. Then the following statements are equivalent:
	
	\begin{enumerate}[fullwidth,itemindent=1em,label=\rm{(\arabic*)}]
		\item $ U \cap VW = (U \cap V)(U \cap W) $;

		\item $ UV \cap UW = U(V \cap W) $.
	\end{enumerate}
	
\end{lemma}

\begin{lemma}[{\cite[Lemma 2.1.6]{Adolfo-2010}}]\label{Sylow}
	Let $ G $ be a $ p $-supersoluble group. Then $ G' $ is $ p $-nilpotent. In particular, if $ O_{p'}(G)=1 $, then $ G $ has a unique Sylow $ p $-subgroup.
\end{lemma}

\begin{lemma}[{\cite[Theorem B]{Skiba-2010}}]\label{JG}
	Let $ \FF $ be a formation. Let  $ E $ be a normal
	subgroup of $ G $ such that  $ F^{*}(E) \leq Z_{\FF}(G) $. Then $ E \leq Z_{\FF}(G) $.
\end{lemma}

\begin{lemma}[{\cite[Lemma 2.11]{Su-2014}}]\label{su}
	Let $ p $ be a prime and $ \FF $ a  solvably saturated formation containing $ \UU $. Suppose that $ E $ is a normal subgroup of $ G $ and $  G/E \in \FF  $.
	
	\begin{enumerate}[fullwidth,itemindent=1em,label=\rm{(\arabic*)}]
		\item\label{su1} If $ E  \leq Z_{\UU}(G) $, then $ G\in \FF $.
		\item\label{su2} If $ E  \leq Z_{p\UU}(G) $, then $ G/O_{p'}(G)\in \FF $.
	\end{enumerate}
\end{lemma}

\begin{lemma}[{\cite[Lemma 2.3]{Su-2014}}]\label{necessity}
	 Let $ p $ be a prime and $ G $ a group such that $ p $ divides the order of $ G $. Then every $ p $-subgroup $ L $ of $ Z_{p\UU}(G) $ satisfies the $ \Pi $-property in $ G $.
\end{lemma}

\begin{lemma}\label{one-of}
	Let  $ N $ be a normal subgroup of  $ G $, $ P\in \mathrm{Syl}_{p}(N) $ and $ T $ the unique minimal normal subgroup of $ G $ contained in $ N $. Suppose that  every subgroup of $ P $ of order $ d $  satisfies the  $ IC $-$ \Pi $-property in $ G $,  where $ d $ is a power of $ p $ and $ 1 < d < |P| $.  Then $ T $ has order $ p $ if one of the following holds:
	
	\begin{enumerate}[fullwidth,itemindent=1em,label=\rm{(\alph*)}]
		\item\label{one} $ |T| = d $;
		\item\label{two} $ T $ is a $ p $-group, $ |T| < d $ and $ T\nleq \Phi(P) $.
	\end{enumerate}	
\end{lemma}

\begin{proof}
	Let $ G_{p} $ be a Sylow $ p $-subgroup of $ G $ containing $ P $. Then $ P=G_{p}\cap N $. If $ |T|=d $ or $ T $ is a $ p $-group with $ |T|<d $, then we can  choose a normal subgroup $ S $ of $ G_{p} $ such that  $ T\leq S\leq P $ and $ |S|=dp $.
	
	\ref{one} Since  $ |T| = d $, we have   $ \Phi(S) \leq T $. If $ \Phi(S) = T $, then $ S $ is cyclic, and so $ |T|=p $, we are done. Hence we may suppose that $ \Phi(S) < T $. Since  $ \Phi(S) \unlhd G_{p} $, $ G_{p} $ has a normal subgroup $ T_{1} $ such that  $ \Phi(S)\leq  T_{1} \leq T  $ and $ |T : T_{1}| = p $. Since $ S/T_{1} $
	is elementary abelian of order $ p^{2} $, there is another maximal subgroup $ H/T_{1} $ of $ S/T_{1} $ such that $ H \not = T $. Then $ H \cap T = T_{1} $. Clearly, $ |H|=d $. If $ [H, G]=1 $, then $ H\leq Z(G) $. Since  $ T $ is the unique minimal normal subgroup of $ G $ contained in $ N $, we deduce that   $ |T|=|H|=p $, as desired.  If $ [H, G]>1 $, then $ T\leq [H, G] $.
	By  hypothesis, $ H $ satisfies the $ IC $-$ \Pi $-property in $ G $. For the $ G $-chief factor $ T/1 $, we see that  $ |G:N_{G}(H\cap T)|=|G:N_{G}(T_{1})| $ is a $ p $-number. Since $ T_{1}=H\cap T\unlhd G_{p} $, we have $ T_{1}\unlhd G $. The minimality of $ T $ implies $ T_{1}=1 $, and hence $ |T|=p $, as desired.
	
	\ref{two} Assume that $ T $ is a $ p $-group, $ |T| < d $ and $ T\nleq \Phi(P) $. Clearly, $ T \nleq \Phi(S) $ and $ T \cap \Phi(S) < T $. Then $G_{p}$ has a normal subgroup $T_{1}$ such that $ T \cap \Phi(S) \leq T_{1} \leq  T  $ and
	 $ |T : T_{1}| = p $. Write $ \overline{S}=S/T_{1}\Phi(S) $. Then $ \overline{S} $ is elementary abelian and $ \overline{T}=T\Phi(S)/T_{1}\Phi(S) $ has order $ p $. There exists a complement $\overline{H} = H/T_{1}\Phi(S) $ for $ \overline{T} $ in $ \overline{S} $, and $ H $ has order $ d $.  Since $ T \not= H $ and
	$ T_{1} \leq H \cap  T $, we know that $ T_{1}=H\cap T\unlhd G_{p} $.  Arguing as in the previous paragraph, we can
	obtain that $ |T| = p $, as desired.
\end{proof}

\section{Proofs}

\begin{theorem}\label{min}
	Let $ P $ be  a normal $ p $-subgroup of $ G $. Then $ P\leq Z_{\UU}(G) $ if $P$ satisfies the following:
\begin{enumerate}[fullwidth,itemindent=1em,label=\rm{(\arabic*)}]
  		\item Every subgroup  of $ P $ of order $ p $  satisfies the  $ IC $-$ \Pi $-property in $ G $;
  		\item If $ p = 2 $ and $ P $ is non-abelian,  we further suppose that every  cyclic subgroup  of $ P $ of order $ 4 $ satisfies the $ IC $-$ \Pi $-property in $ G $.
  		
  	\end{enumerate}
\end{theorem}

\begin{proof}
	Suppose that the result is false and let $ (G, P) $ be a counterexample for which $ |G|+ |P| $ is minimal. We proceed via the following steps.

	\smallskip
	\begin{enumerate}[fullwidth]
		\renewcommand{\labelenumi}{\textbf{Step \theenumi.}}
		\setcounter{enumi}{0}
		\item\label{min1} $ P $ has a unique  maximal $ G $-invariant subgroup, say $ T $. Moreover, $ T\leq Z_{\UU}(G) $ and $ |P/T|>p $.
	\end{enumerate}
	\smallskip

	

	Let $ T $ be  a maximal $ G $-invariant subgroup of $ P $.  It is easy to see that $ (G, T) $ satisfies the hypotheses. By the choice of $ (G,P) $, we have that $ T\leq Z_{\UU}(G) $. If $ |P/T|=p $, then
	$ P\leq Z_{\UU}(G) $, and so $ P\leq Z_{\UU}(G) $, a contradiction. Hence $ |P/T| > p $.  Now assume that $ K $ is a minimal $ G $-invariant subgroup of $ P $ which is different
	from $ T $.   Then  $ P=TK\leq Z_{\UU}(G) $, a contradiction.


	\smallskip
	\begin{enumerate}[fullwidth]
		\renewcommand{\labelenumi}{\textbf{Step \theenumi.}}
		\setcounter{enumi}{1}
		\item\label{min2} The  exponent  of $ P $ is $ p $ or $ 4 $ (when $ P $ is a non-abelian $ 2 $-group).
	\end{enumerate}
	\smallskip
	
	Let $ C $ be a Thompson critical subgroup of $ P $. If $ \Omega(C) < P $, then $ \Omega(C) \leq T  \leq Z_{\UU}(G) $ by Step \ref{min1}. Applying  Lemma \ref{then}, we see that $ P\leq Z_{\UU}(G) $, a contradiction. Hence $ \Omega(C)=P $. By Lemma \ref{exp}, the  exponent  of $ P $ is $ p $ or $ 4 $ (when $ P $ is a non-abelian $ 2 $-group).
	
	\smallskip
	\begin{enumerate}[fullwidth]
		\renewcommand{\labelenumi}{\textbf{Step \theenumi.}}
		\setcounter{enumi}{2}
		\item\label{min3} The  final contradiction.
	\end{enumerate}
	\smallskip
	

	Let $ G_{p} $ be a Sylow $ p $-subgroup of $ G $. Then  $ P/T\cap Z(G_{p}/T)>1 $.   Let $ X/T $ be a subgroup of $ P/T\cap Z(G_{p}/T) $
	of order $ p $. Then we can choose an element $ x\in X\setminus T $. Set $ H=\langle x \rangle $. Then $ X=HT $ and  $ H $ has order $ p $ or $ 4 $ (when $ P $ is a non-abelian $ 2 $-group) by Step \ref{min2}. By hypothesis, $ H $ satisfies the $ IC $-$ \Pi $-property in $ G $.
	
	Assume that $ [H, G] < P $. Since $ [H, G] \unlhd G $,  it follows from Step \ref{min1} that $ [H, G] \leq T $. Hence $ X=HT=H[H, G]T=H^{G}T\unlhd G $. Since $ H^{G}\nleq T $, it follows from Step \ref{min1} that  $ X=P $. Thus $ |P/T=p| $, a contradiction.  Hence we may assume that $ [H, G] = P $. By hypothesis, $ H $ satisfies the $ IC $-$ \Pi $-property in $ G $. For the $G$-chief factor $P/T$, we see that $ |G:N_{G}(HT\cap P)|=|G:N_{G}(X)| $ is a $ p $-number. Note that $ X\unlhd G_{p} $, we have $ X\unlhd G $. By Step \ref{min1}, we conclude  that $ X=P $. Therefore, $ |P/T|=p $, a contradiction.
\end{proof}

\begin{theorem}\label{max}
	Let $ P $ be  a normal $ p $-subgroup of $ G $. If every maximal
	subgroup of $ P $ satisfies the $ IC $-$ \Pi $-property in $ G $, then $ P\leq Z_{\UU}(G) $.
\end{theorem}

\begin{proof}
	Suppose that the result is false and let $ (G, P) $ be a counterexample for which $ |G|+ |P| $ is minimal.

	\smallskip
	\begin{enumerate}[fullwidth]
		\renewcommand{\labelenumi}{\textbf{Step \theenumi.}}
		\setcounter{enumi}{0}
		\item\label{max1} $ P $ has a unique  minimal $ G $-invariant subgroup, say $ T $. Moreover, $ P/T\leq Z_{\UU}(G/T) $ and $ |T|>p $.
	\end{enumerate}
	\smallskip

	Let $ T $ be  a  minimal $ G $-invariant subgroup of $ P $. By Lemma \ref{OVE}\ref{i}, $ (G/T, P/T) $ satisfies the hypotheses of the theorem.   The minimal choice of $ (G, P) $ implies that $ P/T\leq Z_{\UU}(G/T) $.  If $ |T| = p $, then $ P \leq Z_{\UU}(G) $, a
	contradiction. Hence $ |T| > p $.  Assume that $ P $ has another minimal normal subgroup $ K $. With a similar
argument as above, we can get that  $ P/K \leq Z_{\UU}(G/K) $. It follows that  $ P/K=TK/K \leq Z_{\UU}(G/K) $, and so  $ |T| = p $, a contradiction. This shows that $ T $ is the unique  minimal $ G $-invariant subgroup of $ P $.

	\smallskip
	\begin{enumerate}[fullwidth]
		\renewcommand{\labelenumi}{\textbf{Step \theenumi.}}
		\setcounter{enumi}{1}
		\item\label{max2}  $ \Phi(P)>1 $.
	\end{enumerate}
	\smallskip
	
	Assume that $ \Phi(P)=1 $. Then
	there exists  a subgroup $ L $ of $ P $ such that $ P = T\times L $. Let $ G_{p} $ be a Sylow $ p $-subgroup of $ G $. Then $ P\leq G_{p} $. Let $ T_{1} $ be a maximal subgroup of $ T $ such that  $ T_{1}\unlhd G_{p} $. Clearly, $ \mathrm{Core}_{G}(T_{1})=1 $ and $ T_{1}L $ is a maximal subgroup of $ P $. Set $ H=T_{1}L  $. If $ [H, G]=1 $, then $ T_{1}\unlhd G $, a contradiction. Therefore, $ [H, G]>1 $. Since $ [H, G]\unlhd G $, it follows from Step \ref{max1} that $ T\leq [H, G] $.  By the hypothesis, $ H $ satisfies the $ IC $-$ \Pi $-property in $ G $. For the $G$-chief factor $T/1$, we see that  $ |G:N_{G}(H\cap [H, G]\cap T)|=|G:N_{G}(H\cap T)| $ is a $ p $-number. Observe that $ T_{1}\leq H\cap T\leq T $. Since $ T\nleq  H $, we deduce that $ T_{1}=H\cap T $. Now $ T_{1}\unlhd G_{p} $ imples that $ T_{1}\unlhd G $. By Step \ref{max1}, we have $ |T_{1}|=1 $, and thus $ |T|=p $, a contradiction. Therefore, $ \Phi(P)>1 $.

	\smallskip
	\begin{enumerate}[fullwidth]
		\renewcommand{\labelenumi}{\textbf{Step \theenumi.}}
		\setcounter{enumi}{2}
		\item\label{max3}  The  final contradiction.
	\end{enumerate}
	\smallskip
	
	By Step \ref{max1} and Step \ref{max2}, we have $ T\leq \Phi(P) $, and so $ P/\Phi(P)\leq Z_{\UU}(G/\Phi(P)) $.  Then by
	Lemma \ref{Phi}, we conclude that $ P\leq Z_{\UU}(G) $. This final contradiction completes the proof.
\end{proof}

  \begin{proof}[Proof of Theorem \ref{order-d}]
  		Suppose that the result is false and let $ (G, P) $ be a counterexample for which $ |G|+ |P| $ is minimal.

  		\smallskip
  		\begin{enumerate}[fullwidth]
  			\renewcommand{\labelenumi}{\textbf{Step \theenumi.}}
  			\setcounter{enumi}{0}
  			\item\label{order-d1}  $ p<d<\frac{|P|}{p} $.
  		\end{enumerate}
  		\smallskip
  		
  		 By Theorems \ref{min} and \ref{max}, Step \ref{order-d1} holds.
  		
  		 \smallskip
  		 \begin{enumerate}[fullwidth]
  		 	\renewcommand{\labelenumi}{\textbf{Step \theenumi.}}
  		 	\setcounter{enumi}{1}
  		 	\item\label{order-d2}  $ T<P $.
  		 \end{enumerate}
  		 \smallskip

  		Let $ T $ be  a minimal $ G $-invariant subgroup of $ P $.  Assume that $ T = P $. Let $ H $ be a subgroup of
  		$ P $ of order $ d $.   Since $ H[H, G] = H^{G} \leq P $ and $ [H, G] \unlhd  G $, we have that $ [H,G] = P $ or $ 1 $. If $ [H, G] = 1 $, then $ 1 < H \unlhd G $, a contradiction. Thus $ [H, G]=P $. By the hypothesis, $ H \cap  [H, G] $ satisfies $ \Pi $-property in $ G $. Since $ [H, G]=P=T $, we have that  $ |G:N_{G}(H)| $ is a $ p $-number.  Hence $ H\unlhd G $. The  minimality of $ T $ implies that $ H=1 $, a contradiction. Therefore, $ T < P $.

  		\smallskip
  		\begin{enumerate}[fullwidth]
  			\renewcommand{\labelenumi}{\textbf{Step \theenumi.}}
  			\setcounter{enumi}{2}
  			\item\label{order-d3}  $ |T|<d $.
  		\end{enumerate}
  		\smallskip

  		If $ |T| > d $, then  $ (G, T) $ satisfies the hypotheses, and thus $ T \leq Z_{\UU}(G) $ by the minimal choice of $ (G, P) $. The  minimality of $ T $ yields  that $ d = 1 $, a contradiction.
  		
  		 Suppose that $ |T| =d $. Now Step \ref{order-d1} guarantees that
  		$ |T| > p $ and $ |P : T| =\frac{|P|}{d} > p $. By a similar  argument as above, it is easy to see that $ P/T $ is a chief
  		factor of $ G $.  Let $ A/T $ be a normal subgroup of $ G_{p}/T $ contained in $ P/T $  with  order $ p $. Certainly, $ A < P $
  		and $ A \ntrianglelefteq G $. Since $ T $ is noncyclic, we see that  $ A $ is noncyclic, so that there exists a maximal subgroup $ A_{1} $ of $ A $ such that
  		$ A = A_{1}T $. Obviously
  		$ |A_{1}|=|T|=d $. By the hypothesis, $ A_{1} $ satisfies the $ IC $-$ \Pi $-property in $ G $. Since  $ T\nleq A_{1} $ and $ T $ is a minimal normal subgroup of  $ G $, we know that  $ T\cap \mathrm{Core}_{G}(A_{1}) = 1 $. If $ \mathrm{Core}_{G}(A_{1})>1 $, then   $ A = T\mathrm{Core}_{G}(A_{1}) \unlhd G $ since $ T $ is a maximal subgroup of $ A $, a contradiction. Hence $ \mathrm{Core}_{G}(A_{1})=1 $. Since  $ T \cap A_{1} > 1 $ and $ T $ is a minimal normal subgroup of $ G $, we have  $ 1 < [T\cap A_{1}, G] \leq  [T, G] \leq  T $. Thus  $ T = [T \cap A_{1}, G] \leq  [A_{1}, G] \leq  P $. There holds either $ [A_{1}, G] = T $ or $ [A_{1}, G] = P $. If $ [A_{1}, G] = T $, then
  		$ (A_{1})^{G}= A_{1} [A_{1}, G] = A_{1}T = A $. It follows that $ A\unlhd G $, a contradiction.  So $ [A_{1}, G] = P $. Since $ A_{1} $ satisfies the $ IC $-$ \Pi $-property in $ G $, we see that $ |G:N_{G}(A_{1}\cap T)| $ is a $ p $-number. As $ A_{1}\unlhd G_{p} $, we conclude  that $ A_{1}\cap  T\unlhd G $. The minimality of $ T $ implies  that $ A_{1}\cap T=1 $, and hence $ d=|T|=p $,  contrary to Step \ref{order-d1}.

  		\smallskip
  		\begin{enumerate}[fullwidth]
  			\renewcommand{\labelenumi}{\textbf{Step \theenumi.}}
  			\setcounter{enumi}{3}
  			\item\label{order-d4}  $ P/T\leq Z_{\UU}(G/T) $.
  		\end{enumerate}
  		\smallskip

  		Now $ |T| < d $ by Step \ref{order-d3}. If $ p > 2 $, or $ p = 2 $ and $ d /|T| > 2 $, or $ p = 2 $ and $  P/T $ is an abelian $ 2 $-group, then the hypotheses
  		hold for $ (G/T , P/T) $ by Lemma \ref{OVE}\ref{i}.   The minimal choice of $ (G, P) $ implies that $ P/T\leq  Z_{\UU}(G/T) $, as  desired.


  		Now suppose  that $ d /|T| = 2 $ and $ P/T $ is a non-abelian $ 2 $-group. By hypothesis, every  subgroup of $ P $ of order
  		$ 2|T| $ satisfies the $ IC $-$ \Pi $-property in $ G $.  By Lemma \ref{OVE}\ref{i}, every cyclic subgroup of $ P/T $ of order $ 2 $ satisfies the
  		$ IC $-$ \Pi $-property in $ G/T $.  Let $ U/T $ be any cyclic group of $ P/T $ of order $ 4 $. Assume that $ T\leq \Phi(U) $.   Then $ U $
  		is cyclic and thus $ |T| = 2 $ and $ d = 4 $. Hence any subgroup of $ P $ of order $ 4 $ satisfies the $ IC $-$ \Pi $-property in $ G $. Let $L$ be a subgroup of $P$ of order $2$. We argue that $L$ satisfies the $ IC $-$ \Pi $-property in $ G $. It is no loss to assume that $L\not =T$. Then by Lemma \ref{satifies}, $L$ satisfies the $ IC $-$ \Pi $-property in $ G $, as claimed. By Theorem \ref{min}, it follows that $ P\leq Z_{\UU}(G) $, a contradiction. So $ U $ has a maximal subgroup $ U_{1} $ such that $ U=U_{1}T $. It is easy to see that $ |U_{1}| = 2|T| = d $ and thus $ |U/T|= |U_{1}T/T|=|U_{1}/(U_{1}\cap T)|=4 $. If $ U_{1}\cap T=1 $, then $ d=|U_{1}|=4 $.  By hypothesis, every  subgroup of $ P $ of order $ 4 $ satisfies the $ IC $-$ \Pi $-property in $ G $. With the same argument as above, we can conclude that $P\leq Z_{p\UU}(G)$, a contradiction. Therefore, $ U_{1}\cap T>1 $.  Observe that $ 1 < [U_{1}\cap T, G] \leq [T, G] \leq T $. The minimality of $ T $ implies that
  		 $ T = [T, G] = [U_{1} \cap  T, G] \leq [U_{1}, G] $. By \cite[5.15(ii)]{Robinson-1995}, we have $ U \cap [U, G] = U_{1}T \cap [U_{1}T, G] =U_{1}T\cap [U_{1}, G][T, G] = U_{1}T \cap [U_{1}, T]= (U_{1} \cap [U_{1}, G])T $. Since $U_{1}\cap [U_{1}, G]$ satisfies the $ \Pi $-property in $ G $, it follows from Lemma \ref{over} that $(U_{1} \cap [U_{1}, G])T/T=U/T\cap [U/T,G/T]$ satisfies the $ \Pi $-property in $ G/T $.   This means that every cyclic subgroup of $ P/T $ of order $ 4 $ satisfies the
  		 $ IC $-$ \Pi $-property in $ G/T $. By Theorem \ref{min}, $ P/T\leq Z_{\UU}(G/T) $, as desired.


  		 \smallskip
  		 \begin{enumerate}[fullwidth]
  		 	\renewcommand{\labelenumi}{\textbf{Step \theenumi.}}
  		 	\setcounter{enumi}{4}
  		 	\item\label{order-d5}  The final contradiction.
  		 \end{enumerate}
  		 \smallskip
  		
  		Since $ |T|<d<\frac{|P|}{p} $ and $P/T\leq Z_{p\UU}(G/T)$, there exists a normal
  		subgroup $ K $ of $ G $ such that $ T < K < P $ and $|K| > d $.  Clearly, $ (G, K) $ satisfies the hypotheses and so
  		$ K \leq  Z_{\UU}(G) $ by the minimal choice of $ (G, P) $.  This yields  that $ |T| = p $,  and so $ P \leq  Z_{\UU}(G) $ by Step \ref{order-d4}, a contradiction. The proof is  now  complete.
  \end{proof}

\begin{proof}[Proof of Theorem \ref{minimal}]
	Assume that this theorem is not true and let $ (G , N) $ be a minimal counterexample for which $ |G|+|N| $ is minimal.

	\smallskip
	\begin{enumerate}[fullwidth]
		\renewcommand{\labelenumi}{\textbf{Step \theenumi.}}
		\setcounter{enumi}{0}
		\item\label{minimal1}  $ O_{p'}(N)=1 $.
	\end{enumerate}
	\smallskip

	By Lemma \ref{OVE}\ref{ii}, the hypotheses hold for $ (G/O_{p'}(N), N/O_{p'}(N)) $. The minimal choice of $ ( G , N ) $ implies that $ O_{p'}(N) = 1 $.
	
	\smallskip
	\begin{enumerate}[fullwidth]
		\renewcommand{\labelenumi}{\textbf{Step \theenumi.}}
		\setcounter{enumi}{1}
		\item\label{minimal2}  $ N $ is not $ p $-soluble.
	\end{enumerate}
	\smallskip

	Assume that $ N $ is $ p $-soluble. Since $ O_{p'}(N)=1 $, it follows from \cite[Lemma 2.10]{Adolfo-2009} that $ F^{*}_{p}(N) = F_{p}(N) =
	O_{p}(N) $ .  By Theorem \ref{min}, we have $ F^{*}_{p}(N) =O_{p}(N) \leq Z_{\UU}(G) \leq Z_{p\UU}(G) $. Hence $ N\leq Z_{p\UU}(G) $ by \cite[Lemma 2.13]{Su-2014}, a contradiction.

	\smallskip
	\begin{enumerate}[fullwidth]
		\renewcommand{\labelenumi}{\textbf{Step \theenumi.}}
		\setcounter{enumi}{2}
		\item\label{minimal3}  $ N $ has a unique maximal $ G $-invariant subgroup $ U $. Moreover, $ 1<U\leq Z_{p\UU}(G) $.
	\end{enumerate}
	\smallskip

	Assume that  $ N $ is a minimal normal subgroup of $ G $. Let $ G_{p} $ be a Sylow $ p $-subgroup of $ G $ containing $ P $. Then $ G_{p}\cap N =P $.  Let  $ H $ be a cyclic subgroup of $ P $ of order $ p $ such that $ H\unlhd G_{p} $.  If $ [H, G]=1 $, then the minimality of $ N $ implies that $ |H|=|N|=p $, and thus $ N\leq Z_{\UU}(G) $, a contradiction.  Therefore,  $ [H, G]= N $. By hypothesis, $ H $ satisfies the  $ IC $-$ \Pi $-property in $ G $. For the  $G$-chief factor $N/1$, we see that $ |G:N_{G}(H\cap N)|=|G:N_{G}(H)| $ is a $ p $-number. As a consequence, $ H\unlhd G $. The minimality of $ N $ yields that $ N=H $, and so $ N\leq Z_{\UU}(G) $, a contradiction.  Hence $ N $ has a maximal $ G $-invariant subgroup $ U\not =1 $. It is easy to see  that the hypotheses also hold for $ ( G , U ) $ and the minimal choice of $ ( G , N) $ yields that $ U\leq  Z_{p\UU}(G) $. If $ R $ is a maximal $ G $-invariant subgroup of $ N $ which is different from $ U $, then  $ R\leq Z_{p\UU}(G) $, and thus $ N=UR\leq Z_{p\UU}(G) $, a contradiction. Consequently, $ U $ is the unique maximal $ G $-invariant subgroup of $ N $.

	\smallskip
	\begin{enumerate}[fullwidth]
		\renewcommand{\labelenumi}{\textbf{Step \theenumi.}}
		\setcounter{enumi}{3}
		\item\label{minimal4}  Every cyclic subgroup of $ P $ of order $ p $ or $ 4 $ (when $ P $ is a non-abelian $ 2 $-group)  is contained in $ U $.
	\end{enumerate}
	\smallskip

	Assume that there exists a cyclic subgroup $ H $ of $ P $ of order $ p $ or $ 4 $ (when $ P $ is not abelian)  such that $ H\nleq U $. By Step \ref{minimal3}, $ H[H, G]U=H^{G}U=N $.   Assume that $ [H, G]< N $.  The uniqueness of $ U $ yields  that $[H, G]\leq U$, and so $ HU=N $. By Step \ref{minimal3}, $ N $ is $ p $-soluble, which contradicts Step \ref{minimal2}. Hence $ [H, G]=N $.
	By hypothesis, $ H $ satisfies the $ IC $-$ \Pi $-property in $ G $. For the $ G $-chief factor $ N/U $, we see that $ |G:N_{G}(HU)| $ is a $ p $-number. Thus $ G=N_{G}(HU)P $. By Step \ref{minimal3}, we see that $ N=(HU)^{G}=H^{P}U $. By Step \ref{minimal3}, it follows that $ N $ is $ p $-soluble, which contradicts Step \ref{minimal2}. Therefore, Step \ref{minimal4} holds.
	
	\smallskip
	\begin{enumerate}[fullwidth]
		\renewcommand{\labelenumi}{\textbf{Step \theenumi.}}
		\setcounter{enumi}{4}
		\item\label{minimal5}  The final contradiction.
	\end{enumerate}
	\smallskip

	Since $ U\leq Z_{p\UU}(G) $, it follows from Step \ref{minimal4} and \cite[Lemma 2.12]{Su-2014} that $ N $ is $ p $-supersoluble. In particular, $ N $ is $ p $-soluble,  contrary to Step \ref{minimal2}.  This  final contradiction completes the proof.
\end{proof}

\begin{proof}[Proof of Theorem \ref{maximal}]
	Assume that this theorem is not true and let $ (G, N) $ be a counterexample for which $ |G|+|N| $ is minimal. Then $ |P|\geq p^{2} $ and $ N\nleq Z_{p\UU}(G) $. Let $ G_{p} $ be a Sylow $ p $-subgroup of $ G $ containing $ P $. Then $ G_{p}\cap N=P $.   We will obtain  a contradiction in several steps.
	
	\smallskip
	\begin{enumerate}[fullwidth]
		\renewcommand{\labelenumi}{\textbf{Step \theenumi.}}
		\setcounter{enumi}{0}
		\item\label{maximal1}  $ O_{p'}(N)=1 $.
	\end{enumerate}
	\smallskip

		By Lemma \ref{OVE}\ref{ii}, the hypotheses hold for $ (G/O_{p'}(N), N/O_{p'}(N)) $. The minimal choice of $ ( G , N ) $ implies that $ O_{p'}(N) = 1 $.
		
	
	\smallskip
	\begin{enumerate}[fullwidth]
		\renewcommand{\labelenumi}{\textbf{Step \theenumi.}}
		\setcounter{enumi}{1}
		\item\label{maximal2}  $ |P|\geq p^{3} $.
	\end{enumerate}
	\smallskip

	If $ |P|=p^{2} $,  then every subgroup of $ P $  of order $ p $ satisfies the  $ IC $-$ \Pi $-property in $ G $. By Theorem \ref{minimal}, we conclude that $ N\leq Z_{p\UU}(G) $, a contradiction. Hence Step \ref{maximal2} holds.

	\smallskip
	\begin{enumerate}[fullwidth]
		\renewcommand{\labelenumi}{\textbf{Step \theenumi.}}
		\setcounter{enumi}{2}
		\item\label{maximal3}  Let $ T $ be a minimal subgroup of $ G $ contained in $ N $. Then   either $ |PT/T|_{p} = p $ or $ N/T \leq  Z_{p\UU}(G/T) $.
	\end{enumerate}
	\smallskip
	
	By Step \ref{maximal1}, we have $ p\big | |T| $. Clearly, $ PT/T $ is a Sylow
	$ p $-subgroup of $ N/T $.  Let $ L/T $ be a maximal subgroup of $ PT/T $. Then $ L = L \cap PT = (L \cap P)T $. Set
	$ L \cap P = P_{1}  $. Then $ |P : P_{1}| = |PT/T : L/T| = p $. Therefore, $ P_{1} $ is a maximal subgroup of  $ P $. Since $ P_{1} \cap T = P \cap T $ is a Sylow $ p $-subgroup of $ T $, it follows that
    \begin{align*}
	|T \cap P_{1}[P_{1}, G]|_{p} = |T|_{p} = |T \cap P_{1}| = |(T \cap P_{1})(T \cap [P_{1}, G])|_{p}
	\end{align*} and \begin{align*}
	|T\cap P_{1}[P_{1}, G]|_{p'} &=\frac{|P_{1}[P_{1}, G]|_{p'}|T|_{p'}}{|TP_{1}[P_{1}, G]|_{p'}} \\
	&= \frac{|[P_{1}, G]|_{p'}|T|_{p'}}{|T[P_{1}, G]|_{p'}} \\
	&= |T\cap [P_{1}, G]|_{p'}\\
	&= |(T \cap P_{1})(T \cap [P_{1}, G])|_{p'}.
	\end{align*}
	Therefore, $ T\cap P_{1}[P_{1}, G] = (T \cap P_{1})(T\cap [P_{1}, G]) $. By Lemma \ref{equivalent}, we have $ P_{1}T \cap [P_{1}, G]T =
	(P_{1} \cap [P_{1}, G])T $.  Thus
	\begin{align*}
	L/T \cap [L/T, G/T] &=P_{1}T/T \cap [P_{1}T/T, G/T]  \\
	&= (P_{1}T \cap [P_{1}, G]T)/T \\
	&= (P_{1} \cap [P_{1}, G])T/T.
	\end{align*}
	By Lemma \ref{over}, $L/T \cap [L/T, G/T]=(P_{1} \cap [P_{1}, G])T/T$  satisfies the  $ \Pi $-property in $ G/T $. This means that $ L/T $ satisfies the $ IC $-$ \Pi $-property in $ G/T $. Therefore $ (G/T, N/T) $ satisfies the hypotheses of the theorem. The minimal choice $ (G, N) $ yields that $ |PT/T|=p $ or $ N/T\leq Z_{p\UU}(G/T) $.

 	\smallskip
 	\begin{enumerate}[fullwidth]
 		\renewcommand{\labelenumi}{\textbf{Step \theenumi.}}
 		\setcounter{enumi}{3}
 		\item\label{maximal4} $ T $ is the unique minimal normal subgroup of $ G $ contained in $ N $.
 	\end{enumerate}
 	\smallskip

	Let $ T_{1} $ and $ T_{2} $ be two different minimal normal subgroups of $ G $ contained in $ N $.   Then $ N/T_{i}\leq  Z_{p\UU}(G/T_{i}) $ or $ |PT_{i}/T_{i}|=p $ for $ i=1, 2 $ by Step \ref{maximal3}. In view of Step \ref{maximal1}, we have $ p\big | |T_{i}| $ for $ i=1, 2 $.  First suppose
	that $ N/T_{1}\leq Z_{p\UU}(G/T_{1}) $ and $ N/T_{2}\leq Z_{p\UU}(G/T_{2}) $. Then $ N $ is $ p $-supersoluble. By Lemma \ref{Sylow}, we have $ P\unlhd N $, and so $ P\unlhd G $. Applying Theorem \ref{max}, we deduce that  $ P\leq  Z_{\UU}(G) $, and  thus $ N\leq  Z_{p\UU}(G) $, a contradiction. Secondly,
	without loss of generality, we may suppose that $ N/T_{1}\leq  Z_{p\UU}(G/T_{1}) $ and  $ |PT_{2}/T_{2}|=p $. Observe  that $ T_{2}T_{1}/T_{1} $ is  a minimal normal subgroup of $ G/T_{1} $ contained  in $ N/T_{1} $ and $ p\big | |T_{2}| $,  we have $ |T_{2}|=|T_{2}T_{1}/T_{1}|=p $, and so $ |P|=p^{2} $,  contrary to Step \ref{maximal2}. Lastly suppose
	that $ |PT_{1}/T_{1}|=p $ and $ |PT_{2}/T_{2}|=p $.  Note that  $ (P\cap  T_{1})\cap (P\cap T_{2})=1 $ and  $ |P/(P\cap T_{i})|=p $ for $ i=1, 2 $. Hence $ |P\cap T_{1}|=|P\cap  T_{2}|=p $ and $ |P|=p^{2} $,  contrary to Step \ref{maximal2}.
	
	\smallskip
	\begin{enumerate}[fullwidth]
		\renewcommand{\labelenumi}{\textbf{Step \theenumi.}}
		\setcounter{enumi}{4}
		\item\label{maximal5} $ T\leq O_{p}(N) $.
	\end{enumerate}
	\smallskip


    Assume that $ T\nleq O_{p}(N) $.  Then $ O_{p}(N)=1 $ by Step \ref{maximal4}. Let $ H $ be a maximal subgroup of $ P $ such that $ H\unlhd G_{p} $.
    If $ [H, G]=1 $, then $ H\unlhd G $, and so  $ O_{p}(N)>1 $, a contradiction.  Hence we  can assume that $ [H, G]>1 $. By Step \ref{maximal4}, we  have $ T\leq  [H, G] $. By  hypothesis, $ H $ satisfies the  $ IC $-$ \Pi $-property in $ G $. For the $ G $-chief factor $ T/1 $, we see that  $ |G:N_{G}(H\cap T)| $ is  a $ p $-number. As a consequence,  $ H\cap T\unlhd G $. Hence $ H\cap T\leq O_{p}(N)= 1 $, and so $ |P\cap T|=|G_{p}\cap T|=p $. Note that $ |P|\geq  p^{3} $, we can choose a maximal subgroup  $ L $ of $ P $ such that $ P\cap T\leq L\unlhd G_{p} $.  Clearly, $ [L, G]>1 $. By Step \ref{maximal4}, we know $ T\leq [L, G] $. By hypothesis, $ L $ satisfies the $ IC $-$ \Pi $-property in $ G $.
    Arguing as above, we can deduce  that $ L\cap T\leq O_{p}(G)=1 $,
    and thereby $ P \cap T=L\cap T = 1 $, contrary to Step \ref{maximal1}. Thus Step \ref{maximal4} follows.

    \smallskip
    \begin{enumerate}[fullwidth]
    	\renewcommand{\labelenumi}{\textbf{Step \theenumi.}}
    	\setcounter{enumi}{5}
    	\item\label{maximal6} $ T\nleq \Phi(P) $.
    \end{enumerate}
    \smallskip

	Assume that $ T\leq \Phi(P) $. By Step \ref{maximal4}, either $ |P/T| = p $ or $ N/T\leq  Z_{p\UU}(G/N) $. If $ |P/T| = p $, then $ P $ is  cyclic of order $ p^{2} $, which contradicts Step \ref{maximal2}.   Therefore $ N/T\leq  Z_{p\UU}(G/N) $. Since $ T\leq  \Phi(P) $, we have $ T\leq  \Phi(N) $. It follows from \cite[Lemma 2.10]{Su-2014} that $ N \leq Z_{p\UU}(G) $, a contradiction.

	\smallskip
	\begin{enumerate}[fullwidth]
		\renewcommand{\labelenumi}{\textbf{Step \theenumi.}}
		\setcounter{enumi}{6}
		\item\label{maximal7} The final contradiction.
	\end{enumerate}
	\smallskip
	
	
	Since $ T\nleq \Phi(P) $, we see  $ T\cap \Phi(P)<T $. As $ T\cap \Phi(P)\unlhd G_{p} $, $ T $ has a maximal subgroup $ T_{1} $ such that $ T\cap \Phi(P)\leq  T_{1}\unlhd G_{p} $. Set $ \overline{P}=P/T_{1}\Phi(P) $.
	Note that $ \overline{T}=T\Phi(P)/T_{1}\Phi(P) $ has  order $ p $ and   $ \overline{P} $ is elementary abelian. There exists a complement $ \overline{U}= U/T_{1}\Phi(P) $ for $ \overline{T} $ in $ \overline{P} $. Clearly, $ U $ is a maximal subgroup of $ P $.  Since  $ T_{1}\leq U\cap T<T $, it follows that $ U\cap T=T_{1}\unlhd G_{p} $. We argue that $ T $ has order $ p $.
	If $ [U, G]=1 $, then $ U\leq Z(G) $, and so  $ T_{1}=U\cap T\unlhd G $. The minimality of $T$ implies that  $ |T_{1}|=1 $, and so $ |T|=p $, as claimed. If $ [U, G]>1 $, then $ T\leq [U, G] $ by Step \ref{maximal4}. By hypothesis, $ U $  satisfies the  $ IC $-$ \Pi $-property in $ G $. For the $ G $-chief factor $ T/1 $, we see that $ |G:N_{G}(U\cap T)| $ is a $ p $-number. Since $ U\cap T=T_{1}\unlhd G_{p} $, we have $ T_{1}\unlhd  G $. The minimality of $ T $ yields  that $ |T|=p $, as  claimed.
	By Step \ref{maximal3}, there holds either $ |P/T|=p $ or $ N/T\leq Z_{p\UU}(G/T) $. If $ |P/T|=p $, then $ |P|=p^{2} $, which contradicts Step \ref{maximal2}.  If $ N/T\leq Z_{p\UU}(G/T) $, then $ N\leq Z_{p\UU}(G) $, a  contradiction. Our proof  is now complete.
\end{proof}

We can now prove our main result.

\begin{proof}[Proof of Theorem \ref{main-result}]
	Assume that this result is not true and let $ (G, N) $ be a counterexample such that $ | G |+|N| $ is minimal. Let $ G_{p} $ be a Sylow  $ p $-subgroup of $ G $ such that $ P\leq G_{p} $. Then $ G_{p}\cap X=P $.

	\smallskip
	\begin{enumerate}[fullwidth]
		\renewcommand{\labelenumi}{\textbf{Step \theenumi.}}
		\setcounter{enumi}{0}
		\item\label{main-result1} $ X=N $ and $ O_{p'}(N)=1 $.
	\end{enumerate}
	\smallskip

	Suppose that $ X < N $. Then consider the pair $ (G, X) $. Clearly, $ F_{p}^{*}(X)\leq X\leq X  $.  It is easy to see that $ (G, X) $ satisfies the hypotheses. By the minimal choice of $ (G, N) $, we have
	$ X\leq  Z_{p\UU}(G)  $. Hence $ F_{p}^{*}(N)\le  Z_{p\UU}(G) $. By \cite[Lemma 2.13]{Su-2014}, we have $ N\leq Z_{p\UU}(G) $, a contradiction. Thus $ X=N $.
	
	By Lemma \ref{OVE}\ref{ii}, we know that $ (G/O_{p'}(N), N/O_{p'}(N)) $ satisfies the hypotheses of the theorem, hence
	the minimal choice of $ ( G , N ) $ yields  that $ O_{p'}(N) = 1 $.



	

	
	

	 \smallskip
	 \begin{enumerate}[fullwidth]
	 	\renewcommand{\labelenumi}{\textbf{Step \theenumi.}}
	 	\setcounter{enumi}{1}
	 	\item\label{main-result2} $ p<d<\frac{|P|}{p} $.
	 \end{enumerate}
	 \smallskip
	
	 By Theorems \ref{minimal} and \ref{maximal}, Step \ref{main-result2} holds.

	 \smallskip
	 \begin{enumerate}[fullwidth]
	 	\renewcommand{\labelenumi}{\textbf{Step \theenumi.}}
	 	\setcounter{enumi}{2}
	 	\item\label{main-result3} Let $ T $ be an arbitrary minimal normal subgroup of $ G $ contained in $ N $. Then $ T\leq O_{p}(N) $.
	 \end{enumerate}
	 \smallskip

	 Assume that $ |P\cap T|<d $. Then there exists a subgroup $ H $ of $ P $ of order $ d $ such that $ 1<P\cap T< H\unlhd G_{p} $. By Step \ref{main-result1}, we  have $ 1<P\cap T=H\cap T $. 	
	 If $ [P\cap T, G]=1 $, then $ P\cap T\leq Z(G) $.  The minimality of $ T $ yields  that $ P\cap T=T $, and so $ T\leq O_{p}(N) $, as wanted. Hence we  may suppose that $ [P\cap T, G]>1 $. Clearly, $ 1<[H\cap T, G]=[P\cap T, G]\leq T $. Hence $ 1<[H\cap T, G]  \leq T\cap [H, G]\leq T $. The minimality of $ T $ implies that $ T\cap [H, G]=T $. By hypothesis, $ H $ satisfies the  $ IC $-$ \Pi $-property in $ G $. For the $ G $-chief factor $T/1$, $ |G:N_{G}(H\cap T)| $ is  a $ p $-number. Since $ H\unlhd G_{p} $, we have $ 1<H\cap T\unlhd G $.  As a consequence,  $ H\cap T=T $, and so $ T\leq O_{p}(N) $, as wanted.
	
	 Assume that $ |P\cap T|\geq d $. Let $ L $ be a subgroup of $ P $ of  order $ d $ such that $ L\leq P\cap T $. If $ [L, G]=1 $, then $L\leq Z(G)$. The minimality of $T$ forces  that $p=|L|=|T|=d$, contrary to Step \ref{main-result2}. Hence we can suppose  that $ 1<[L, G] $. The minimality of $ T $ yields  that $ [L, G]=T $. By hypothesis, $ H $ satisfies the  $ IC $-$ \Pi $-property in $ G $. Thus $ |G:N_{G}(H\cap T)|=|G:N_{G}(H)| $ is  a $ p $-number. As a consequence, $ H\unlhd G $. Therefore $ T=H\leq O_{p}(N) $, as wanted.

	
	\smallskip
	\begin{enumerate}[fullwidth]
		\renewcommand{\labelenumi}{\textbf{Step \theenumi.}}
		\setcounter{enumi}{3}
		\item\label{main-result4} $ |T|<d $, $ N/T\leq Z_{p\UU}(G/T) $ and $ T $ is the unique minimal normal subgroup of $ G $ contained in $ N $.
	\end{enumerate}
	\smallskip

    By Step \ref{main-result3}, we know that $T$ is a $p$-group. If $ |T|>d $, then $ T\leq Z_{\UU}(G) $ by Theorem \ref{order-d}, and thus $ |T|=p $. As a consequence, $ d=1 $, a contradiction. Now  assume that $ |T|=d $. Then $ |O_{p}(N)| \geq d $. If $ |O_{p}(N)| > d $, then by Theorem \ref{order-d}, we have
	 $ O_{p}(N) \leq Z_{\UU}(G) $ and so $ d = |T| = p $, contrary to Step \ref{main-result2}. Thus  $ |O_{p}(N)| = d $. By Step \ref{main-result3}, we see that $ T=O_{p}(N) $ is the unique minimal normal subgroup of $ G $ contained in $ N $. By Lemma \ref{one-of}\ref{one}, it follows that $ |T|=p $, which contradicts Step \ref{main-result2}. Thus $ |T|<d $.
	
	 If $ p > 2 $, or $ p = 2 $ and $ d /|T| > 2 $, or $ p = 2 $ and $ P/T $ is an abelian $ 2 $-group, then the hypotheses
	 hold for $ ( G/T, N/T) $. Hence $ N/T \leq Z_{p\UU}(G/T) $ by the minimal choice of $ (G, N) $. Now assume that $d/|T| = 2$ and $P/T$ is a non-abelian $2$-group.  Then every subgroups of $ P $ of order $ 2|T| $ satisfies the  $ IC $-$ \Pi $-property in $ G $. By Lemma \ref{OVE}\ref{i},  every cyclic subgroup of $ P/T $ of order $ 2 $ satisfies the $ IC $-$ \Pi $-property in $ G/T $.  Let $ A/T $ be any cyclic group of $ P/T $ of order $ 4 $.  Assume that $ T\leq  \Phi(A)  $. Then $ A $ is cyclic and thus $ |T| = 2 $ and $ d = 4 $. Then every subgroup of $ P $ of order $ 4 $ satisfies the $ IC $-$ \Pi $-property in $ G $. Let $L$ be a subgroup of $P$ of order $2$. We argue that $L$ satisfies the $ IC $-$ \Pi $-property in $ G $. It is no loss to assume that $L\not =T$. Then by Lemma \ref{satifies}, $L$ satisfies the $ IC $-$ \Pi $-property in $ G $, as claimed. Applying Theorem \ref{minimal}, it follows that $ N\leq Z_{p\UU}(G) $, a contradiction. Hence $ T\nleq \Phi(A) $ and $ A $ has a maximal subgroup of $ R $ such that $ TR=A $. Clearly,  $ |R| = 2|T| = d $ and thus $ |A/T| = |TR/T| = |R/(R\cap T)| = 4 $. If $ R\cap T=1 $, then every subgroup of $ P $ of order $ 4 $ satisfies the $ IC $-$ \Pi $-property in $ G $. Using the same argument as above, every subgroup of order $ 2 $ satisfies the $ IC $-$ \Pi $-property in $ G $.  By Theorem \ref{minimal}, $ N\leq  Z_{p\UU}(G) $, a contradiction.
	 Therefore, $ R\cap T>1 $.  The minimality of $ T $ implies that  $ 1 < [R\cap T, G] \leq [T, G] \leq T $. Hence
	 $ T = [T, G] = [R \cap  T, G] \leq [R, G] $. By \cite[5.15(ii)]{Robinson-1995}, we have $ A\cap [A, G] = RT \cap [RT, G] =RT\cap [R, G][T, G] = RT \cap [R, G]= (R \cap [R, G])T $. By  Lemma \ref{over}, we know that $ A/T $ satisfies the $ IC $-$ \Pi $-property in $ G/T $.  This means that every cyclic subgroup of $ P/T $ of order $ 4 $ satisfies the
	 $ IC $-$ \Pi $-property in $ G/T $. By Theorem \ref{minimal}, $ N/T\leq Z_{\UU}(G/T) $, as desired.
	
	 Assume that $ K $ is another minimal normal subgroup of $ G $. Then $ K\leq O_{p}(N) $  by Step \ref{main-result3}. Using the same argument as above, we have $|K|<d$, $ N/K\leq Z_{p\UU}(G/K) $. Since $ N/T\leq Z_{p\UU}(G/T) $, it follows that  $ N $ is $ p $-supersoluble. By Lemma \ref{Sylow}, we see that $ P\unlhd N $, and thus $ P\unlhd G $. Applying Theorem \ref{order-d}, we can deduce that $ N\leq Z_{p\UU}(G) $, a contradiction. Consequently, $ T $ is the unique minimal normal subgroup of $ G $ contained in $ N $.

	 \smallskip
	 \begin{enumerate}[fullwidth]
	 	\renewcommand{\labelenumi}{\textbf{Step \theenumi.}}
	 	\setcounter{enumi}{4}
	 	\item\label{main-result5} The final contradiction.
	 \end{enumerate}
	 \smallskip
	
	 If  $ T\leq \Phi(P) $, then  $ T\leq \Phi(N) $. Since $ N/T\leq  Z_{p\UU}(G/T)  $, it follows from \cite[Lemma 2.10]{Su-2014} that $ N\leq Z_{p\UU}(G) $, a contradiction. Hence $ T\nleq \Phi(P) $. By Lemma \ref{one-of}\ref{two}, Step \ref{main-result3} and Step \ref{main-result4}, we have $ |T|=p $. Since $ N/T\leq Z_{p\UU}(G/T) $, it follows that  $ N\leq Z_{p\UU}(G) $, and this is our final contradiction.
	\end{proof}

\section{Final remarks and applications}\label{S4}

In this section, we will show that the concept of the $ IC $-$ \Pi $-property can be viewed as a generalization  of many known embedding properties. As a consequence, a large number of known results can follow directly from our main result.

Recall that two subgroups $ H $ and $ K $ of  $ G $ are said to be \emph{permutable}  if $ HK = K H $.
A subgroup $ H $ of $ G $ is said to be \emph{$ S $-permutable}  (or \emph{$\pi$-quasinormal}, \emph{$ S $-quasinormal}) in $ G $ if $ H $ permutes with all Sylow subgroups of $ G $  \cite{Kegel}. A subgroup $ H $ of  $ G $  is said to be \emph{$ X $-permutable}  with a subgroup $ T $ of $ G $  if there is an element $ x \in X $ such that $ HT^{x} = T^{x}H $, where  $ X $ is a   non-empty subset of $ G $ \cite{Guo-2007}.
A subgroup $ H $ of  $ G $ is called a \emph{CAP-subgroup} of $ G $ if $ H $ either covers or avoids every chief factor $ L/K $ of $ G $, that is, $ HL =HK $ or $ H \cap L =H \cap K $ (see \cite[Chapter A, Definition 10.8]{Doerk}).

A subgroup $ H $ of  $ G $ is said to be \emph{$ S $-semipermutable}  \cite{Chen-1987} in $ G $  if $ HG_{p} = G_{p}H $ for any Sylow $ p $-subgroup $ G_{p} $ of $ G $ with $ (p,|H|) = 1 $. A subgroup $ H $ of $ G $ is  said to be \emph{$ SS $-quasinormal}  \cite{Li-2008} in  $ G $ if there is a subgroup  $ B $ of $ G $  such that $ G=HB $ and  $ H $ permutes with every Sylow subgroup of $ B $.

\begin{lemma}\label{p1}
	Let $ H $ be a $ p $-subgroup of  $ G $. Then $ H $  satisfies the $ IC $-$ \Pi $-property in $ G $, if one of the following holds:

	\begin{enumerate}[fullwidth,itemindent=1em,label=\rm{(\arabic*)}]
		\item\label{1} $ H $ is normal in $ G $;
		\item $ H $ is permutable in $ G $;
		\item $ H $ is $ S $-permutable in $ G $;
		\item $ H $ is $ X $-permutable with all Sylow subgroups of $ G $, where $ X $ is a soluble normal subgroup of $ G $;
		\item $ H $ is a CAP-subgroup of $ G $;
		\item\label{6} $ H/H_{G}  \leq Z_{\UU}(G /H_{G}) $;
		
		\item\label{7} $ H $ is a $ p $-group and $ H $ is $ S $-semipermutable in $ G $;
		\item\label{8} $ H $ is a $ p $-group and $ H $ is $ SS $-quasinormal in $ G $.
		
	\end{enumerate}	
\end{lemma}

\begin{proof}
	 Statements \ref{1}-\ref{6} were proved in \cite[Propositions 2.2-2.3]{li-2011}, and the proof of \cite[Proposition 2.4]{li-2011} still works for statement \ref{7}.
	
	 \ref{8} Applying  \cite[Lemma 2.5]{Li-2008}, we see that $ H $ is $ S $-semipermutable in $ G $. Therefore, $ H $ satisfies the $ IC $-$ \Pi $-property in $ G $ by statement \ref{7}.
\end{proof}

\begin{corollary}\label{star}
	Let $ E $ be a normal subgroup of $ G $. Then $ E \leq Z_{\UU}(G) $ if and only if there exists a normal subgroup $ X $ of $ G $ such that $ F^{*}(E) \leq X \leq  E $ and $ X $ satisfies the following:
	
	\begin{enumerate}[fullwidth,itemindent=1em,label=\rm{(\arabic*)}]
		\item Every non-cyclic Sylow subgroup $ P $ of $ X $ has a subgroup $ D $ such that $ 1 < | D | < | P | $ and every subgroup of $ P $ of order $ |D| $ satisfies the $ IC $-$ \Pi $-property in $ G $;
		\item If $ P $ is a non-abelian $ 2 $-group and $ | D | = 2 $, we further suppose that every cyclic subgroup of $ P $ of order $ 4 $ satisfies the $ IC $-$ \Pi $-property in $ G $.
		
	\end{enumerate}	
	
\end{corollary}

\begin{proof}
	 The necessity follows from Lemma \ref{necessity}. We only need to prove the sufficiency. Set $ \pi(X) = \{ p_{1}, p_{2} ,..., p_{s} \} $ and $ p_{1} < p_{2} < \cdots < p_{s} $. Let $ P_{i} $ be a Sylow $ p_{i} $-subgroup of $ G $.  If $ P_{1} $ is cyclic, then $ X $ is $ p_{1} $-nilpotent by \cite[Corollary 5.14]{Isaacs-2008}. Thus $ X\leq Z_{p_{1}\UU}(G) $.  If $ P_{1} $ is not cyclic, then $ X \leq  Z_{p_{1}\UU}(G)  $ by Theorem \ref{main-result}. Since $ p_{1} $ is the smallest prime divisor of $ |X| $, it folows from \cite[Chapter 2, Lemma 5.25(1)]{Guo-2015} that $ X $ is $ p_{1} $-nilpotent.
	 Now $ O_{p_{1}'}(X) $ is the normal Hall $ p_{1}' $-subgroup of $ X $ and $ p_{2} $ is the smallest prime divisor of $ |O_{p_{1}'}(X)| $ . By the similar argument as above, we have $ O_{p_{1}'}(X)\leq Z_{p_{2}\UU}(G) $.  Since $ (p_{1}, |O_{p_{1}'}(X)|)=1 $, we can get that $ X\leq O_{p_{2}\UU}(G) $. Repeat this procedure, finally we have $ X\leq O_{p_{i}\UU}(G) $ for all $ p_{i}\in \pi(X) $. This implies that  $ X\leq O_{\UU}(G) $. As a consequence, $ F^{*}(E)\leq Z_{\UU}(G) $. By Lemma \ref{JG}, we have $ E\leq Z_{\UU}(G) $. It follows from Lemma \ref{su}\ref{su1} that $ G\in \FF $.
\end{proof}

Combining Corollary \ref{star} and Lemma \ref{su}\ref{su1}, we can obtain  the following result.

\begin{corollary}
	Let $ \FF $ be a solubly saturated formation containing $ \UU $ and let $ E $ be a normal subgroup of $ G $ such that $ G/E\in \FF $. Then $ G\in \FF $ if there exists a normal subgroup $ X $ of $ G $ such that $ F^{*}(E) \leq X \leq  E $ and $ X $ satisfies the following:
	
	\begin{enumerate}[fullwidth,itemindent=1em,label=\rm{(\arabic*)}]
		\item Every non-cyclic Sylow subgroup $ P $ of $ X $ has a subgroup $ D $ such that $ 1 < | D | < | P | $ and every subgroup of $ P $ of order $ |D| $ satisfies the $ IC $-$ \Pi $-property in $ G $;
		\item If $ P $ is a non-abelian $ 2 $-group and $ | D | = 2 $, we further suppose that every cyclic subgroup of $ P $ of order $ 4 $ satisfies the $ IC $-$ \Pi $-property in $ G $.
		
	\end{enumerate}
\end{corollary}

The following result  is a consequence of Theorem \ref{main-result} and Lemma \ref{su}\ref{su2}.

\begin{corollary}
	Let $ \FF $ be a solubly saturated formation containing $ \UU $ and let $ E $ be a normal subgroup of $ G $ such that $ G/E\in \FF $. Then $ G/O_{p'}(G)\in \FF $ if there exists a normal subgroup $ X $ of $ G $ such that $ F_{p}^{*}(E) \leq X \leq  E $ and a Sylow $ p $-subgroup $ P $ of $ X $ satisfies the following:

	 \begin{enumerate}[fullwidth,itemindent=1em,label=\rm{(\arabic*)}]
	 	\item Every subgroup of $ P $ of order $ d $ satisfies the $ IC $-$ \Pi $-property in $ G $, where $ d $ is  a power of $ p $ such that $ 1<d<|P| $;
	 	\item If $ P $ is a non-abelian $ 2 $-group and $ d= 2 $, we further suppose that every cyclic subgroup of $ P $ of order $ 4 $ satisfies the $ IC $-$ \Pi $-property in $ G $.
	 \end{enumerate}
\end{corollary}

\end{sloppypar}
\end{document}